\documentclass[12pt,reqno]{amsart}

\usepackage[T1]{fontenc}
\usepackage{lmodern}
\usepackage[text={425pt,665pt},centering]{geometry}
\usepackage{amsmath,amssymb,amsthm}
\usepackage{microtype}
\usepackage{tikz}
\usetikzlibrary{arrows.meta}
\usepackage{float,needspace}
\usepackage[hidelinks]{hyperref}

\numberwithin{equation}{section}
\numberwithin{figure}{section}
\newtheorem{thm}{Theorem}[section]
\newtheorem{prop}[thm]{Proposition}
\newtheorem{lem}[thm]{Lemma}
\theoremstyle{remark}
\newtheorem{rem}{Remark}

\newcommand{\R}{\mathbb R}
\newcommand{\N}{\mathbb N}
\newcommand{\Z}{\mathbb Z}

\newcommand{\ep}{\varepsilon}

\newcommand{\Om}{\Omega}
\newcommand{\cG}{\mathcal G}
\newcommand{\cA}{\mathcal A}
\newcommand{\dist}{\operatorname{dist}}
\definecolor{upper}{RGB}{27,91,139}
\definecolor{lower}{RGB}{164,80,44}

\begin{document}
\title[Pathological viscosity solutions]
{Pathological viscosity solutions of Hamilton--Jacobi equations}
\author[H. V. Tran, Y. Yu]{Hung V. Tran, Yifeng Yu}

\date{}

\thanks{
H. V. Tran is partially supported by NSF grant DMS-2348305. 
}

\address[H. V. Tran]
{
Department of Mathematics, 
University of Wisconsin-Madison, Van Vleck Hall, 480 Lincoln Drive, Madison, Wisconsin 53706, USA}
\email{hung@math.wisc.edu}

\address[Y. Yu]
{
Department of Mathematics, 
University of California at Irvine, 
California 92697, USA}
\email{yifengy@uci.edu}
\begin{abstract}
We construct a continuous function $u$ on the closed unit ball in $\R^4$ and a smooth autonomous Hamiltonian $H:\R^4\to[1,2]$ for which $u$ solves $H(Du)=c$ in the unit ball in the viscosity sense for every $c\in[1,2]$. This reveals a striking and distinctive pathology of viscosity solutions: a fixed function and a fixed Hamiltonian need not determine the right-hand side of the equation.
\end{abstract}
\maketitle

\section{Introduction}

Write $B^d_1=B^d(0,1)$ for $d\in \N$. 
The following is the main result of this paper.

\begin{thm}\label{thm:main}
There are $u\in C(\overline{B^4_1})$ and $H\in C^\infty(\R^4;[1,2])$ such
that $u$ is a viscosity solution of
\begin{equation}\label{eq:main}
H(Du)=c\qquad\text{in }B^4_1
\end{equation}
for every $c\in[1,2]$. 
In particular, it solves both $H(Du)=1$ and $H(Du)=2$. 
The function $u$ is nowhere differentiable in $B^4_1$.
\end{thm}

A fundamental question in the theory of viscosity solutions
\cite{CL83, CEL84, TranHJ} concerns the single-valuedness of
Hamilton--Jacobi operators. In the classical setting, a fixed differentiable function $u$ and a fixed Hamiltonian $H$ uniquely determine $H(Du)$. The viscosity
framework extends the meaning of the equation to continuous functions
through upper and lower test functions; the question is whether this
extension preserves single-valuedness. More precisely, given an open
set $\Omega\subset\mathbb R^n$, $H\in C(\mathbb R^n)$, and
$u,f,g\in C(\Omega)$, does the assumption that $u$ solves both
\[
H(Du)=f
\qquad\text{and}\qquad
H(Du)=g
\]
in $\Omega$ in the viscosity sense imply that $f=g$?

The question was highlighted in its global form in the \emph{User's Guide}
\cite[Remark~2.5]{CIL92}, which systematically developed the modern framework
of viscosity solutions. It was  discussed further in
\cite[Remark~3.6]{DEH96} and \cite[p.~1301]{Yu09}.
For locally Lipschitz $u$, the conclusion follows immediately from
the almost everywhere differentiability of $u$ and the continuity of $f$ and $g$.
For merely continuous $u$, however, upper and lower test functions
need not exist at the same points, and single-valuedness is no longer
automatic. 
Single-valuedness is known in dimension one and for convex or uniformly
continuous Hamiltonians \cite{Evans84, Frankowska86}, and was established
under a weak coercivity condition in \cite{NS96}.

We show that continuity of $H$ alone does not suffice in higher
dimensions. In fact, single-valuedness can fail for an entire interval
of constant right-hand sides: we construct, in dimension four, a
continuous function $u$ and a bounded smooth autonomous Hamiltonian
$H$ such that the same $u$ solves $H(Du)=c$ in the viscosity sense for
every $c\in[1,2]$.

For non-autonomous Hamiltonians, we have a corresponding result in dimension two.

\begin{thm}\label{thm:main2}
There are $F\in C(\overline{B^2_1})$ and $K\in C^\infty(\R^4;[1,2])$ such
that $F$ is a viscosity solution of
\begin{equation}\label{eq:mainK}
K(x,DF)=c\qquad\text{in }B^2_1
\end{equation}
for every $c\in[1,2]$. 
\end{thm}
\subsection{Settings}
Let $O\subset\R^m$ be open and $v\in C(O)$. 
We use the usual first-order superdifferential and subdifferential: For $x\in O$,
\begin{align*}
D^+v(x)&=\left\{p\in\R^m:
 \limsup_{\substack{h\to0\\h\ne0}}
 \frac{v(x+h)-v(x)-p\cdot h}{|h|}\le0\right\},\\
D^-v(x)&=\left\{p\in\R^m:
 \liminf_{\substack{h\to0\\h\ne0}}
 \frac{v(x+h)-v(x)-p\cdot h}{|h|}\ge0\right\},
\end{align*}
If a $C^1$ function $\varphi$ touches $v$ from above at $x$, then $D\varphi(x)\in D^+v(x)$; a lower test gives a gradient in $D^-v(x)$.

For $H(Du)=c$ in $B_1^4$, the viscosity conditions are
\begin{align*}
H(p)&\le c &&\text{ for all $p\in D^+u(x)$ and $x\in B_1^4$,}\\
H(p)&\ge c &&\text{ for all $p\in D^-u(x)$ and $x\in B_1^4$.}
\end{align*}
See \cite[Chapter 1]{TranHJ}.
At a point $x\in B_1^4$ having neither kind of test, both requirements are empty.
The construction exploits this elementary feature. Note that Theorem \ref{thm:main} is equivalent to finding $u\in C(\overline{B^4_1})$  such that
\begin{equation}\label{eq:separation}
    \overline{\bigcup_{x\in B_1^4}D^+u(x)} \ \cap \ \overline{\bigcup_{x\in B_1^4}D^-u(x)}=\emptyset.
\end{equation}
If they are disjoint, $H$ can then be chosen as a smooth separation function for the two sets.

\subsection{Ideas of the construction}

It is useful to keep track of the point of contact as well as the gradient. 
For $v\in C(O)$, set
\[
\cG^\pm(v;O)=\{(x,p)\in O\times\R^m:p\in D^\pm v(x)\}.
\]
Both graphs are nonempty. Indeed, on any closed ball compactly contained
in $O$, sufficiently steep paraboloids can be translated vertically
to touch $v$ from above and below at interior points.  If both $D^+v(x)$ and $D^-v(x)$ are nonempty, then $v$ is differentiable at $x$, and
\[
D^+v(x)=D^-v(x)=\{Dv(x)\}.
\]
In particular,  $\cG^+(v;O)\cap \cG^-(v;O)=\emptyset$ is equivalent to  that $v$ is nowhere differentiable in $O$.

The key step is to construct a nowhere-differentiable planar $F$ with graphs $\cG^\pm(F;O)$ that are relatively closed in $O\times \R^2$, where $\overline{B^2_1}\Subset O\subset\R^2$. 
 Disjointness in the space $(x,p)\in \R^2\times \R^2$ does not by itself separate the two \emph{ranges} of gradients: the two graphs may project onto overlapping sets of $p$'s. To obtain (\ref{eq:separation}), we add additional spatial variables $y\in\R^2$.  For
\[
U(x,y)=F(x)+x\cdot y \qquad \text{where } (x,y)\in O\times\R^2,
\]
the second block of the gradient of every upper or lower test function is $x$. 
Hence the gradient itself remembers the planar point where the test occurs.

\subsection*{Organization of the paper}
We prove the lifting idea in Lemma \ref{lem:lift}.
We then construct the planar function $F$ with disjoint graphs $\cG^\pm(F;O)$ in Section \ref{sec:F}.
The proofs of Theorems \ref{thm:main}--\ref{thm:main2} are given in Section \ref{sec:F}.

\section{From planar graphs to an autonomous Hamiltonian}

\begin{lem}\label{lem:lift}
Let $m\ge1$, $K=\overline{B^m_1}\Subset O\subset\R^m$, and let $F\in C(O)$.
Suppose $\cG^+(F;O)$ and $\cG^-(F;O)$ are relatively closed and disjoint. 
Then, on $\overline{B^{2m}_1}$, the function
\begin{equation}\label{eq:lift}
U(x,y)=F(x)+x\cdot y
\end{equation}
is a viscosity solution of $H(DU)=c$ in $B^{2m}_1$ for every $c\in[1,2]$, with one $H\in C^\infty(\R^{2m};[1,2])$. 
It is nowhere differentiable in $B^{2m}_1$.
\end{lem}

\begin{proof}
We can always regard $U$ as a function on $O\times\R^m$. 
For every $(x,y)$ there, we claim that
\begin{equation}\label{eq:jet-lift}
D^\pm U(x,y)=\{(p+y,x):p\in D^\pm F(x)\}.
\end{equation}
For one inclusion, let $(a,b)\in D^\pm U(x,y)$. 
Increments $(0,k)$ of both signs force $b=x$, since $U(x,y+k)-U(x,y)=x\cdot k$. 
Restriction to increments $(h,0)$ then gives $a-y\in D^\pm F(x)$. 
Conversely, for $p\in D^\pm F(x)$, we have
\[
U(x+h,y+k)-U(x,y)-(p+y)\cdot h-x\cdot k
=F(x+h)-F(x)-p\cdot h+h\cdot k.
\]
The last term is $h\cdot k=O(|(h,k)|^2)$, proving the other inclusion.

Denote by
\begin{equation}\label{eq:A}
\cA^\pm=\{(p+y,x):x,y\in K,\ p\in D^\pm F(x)\}
            \subset\R^{2m}.
\end{equation}
The closed $2m$-dimensional unit ball is contained in $K\times K$, and so \eqref{eq:jet-lift} gives
\[
D^\pm U(x,y)\subset\cA^\pm
             \qquad\text{for }(x,y)\in B^{2m}_1.
\]
As $\cG^+(F;O)$ and $\cG^-(F;O)$ are relatively closed and disjoint, $\cA^\pm$ are both closed and disjoint.
Furthermore, $\cA^\pm$ are both nonempty (see, e.g., \cite[Proposition 1.10]{TranHJ}).

We may therefore choose $H\in C^\infty(\R^{2m};[1,2])$ to be a smooth separation function of $\cA^\pm$ such that,
\begin{equation}\label{eq:distance-H}
H=1 \qquad \text{on ${\cA^+}$} \qquad \mathrm{and} \qquad   H=2\qquad  \text{on ${\cA^-}$}. 
\end{equation}

Finally, disjointness of the two graphs makes $F$ nowhere differentiable. 
If $U$ were differentiable at an interior point, restriction to increments $(h,0)$ in \eqref{eq:lift} would make $F$ differentiable there as well. 
This proves the last assertion.
\end{proof}

\section{Construction of the planar function F}\label{sec:F}

\subsection{Polygonal constructions}
We will start our construction on a fixed polygon.
Then, we will iteratively use this construction in the next subsection.

For $t\in\R$ define
\begin{equation}\label{eq:wave}
d(t)=\dist(t,2\mathbb Z),\qquad
w(t)=\tfrac12\arcsin(1-2d(t)).
\end{equation}
Then, $w$ is continuous, $2$-periodic, smooth in every $(j,j+1)$ for $j\in \Z$, and $|w|\le\pi/4<1$. $w(0)=\frac{\pi}{4}$ and $w(1)=-\frac{\pi}{4}$. On $(0,1)$,
\[
w'(t)=-\frac{1}{2\sqrt{t(1-t)}}.
\]
Also, $w'$ has constant sign in $(j,j+1)$ for $j\in \Z$, and
\begin{equation}\label{eq:wave-slope}
|w'|\ge1,\qquad
\lim_{t\to j}|w'(t)|=\lim_{t\to j+1}|w'(t)|=\infty.
\end{equation}
In fact, for $j\in \Z$ and $|\tau|\le1$,
\begin{equation}\label{eq:wave-cusp}
w(j+\tau)-w(j)=
\begin{cases}
-\arcsin\sqrt{|\tau|},\qquad&j\text{ even},\\
\phantom{-}\arcsin\sqrt{|\tau|},\qquad&j\text{ odd}.
\end{cases}
\end{equation}
The graph of $w$ has a downward square-root cusp at each even integer and an upward square-root cusp at each odd integer (see Figure~\ref{fig:cusps}).

\begin{figure}[H]
\centering
\begin{tikzpicture}[>=Latex,scale=1.05,font=\small]
  \begin{scope}
    \draw[->,gray!75] (-1.2,0)--(1.25,0) node[right] {$t$};
    \draw[->,gray!75] (0,-.04)--(0,1.85);
    \draw[upper,very thick,domain=-1:-.002,samples=70]
      plot (\x,{1.45-.78*sqrt(abs(\x))});
    \draw[upper,very thick,domain=.002:1,samples=70]
      plot (\x,{1.45-.78*sqrt(abs(\x))});
    \fill[upper] (0,1.45) circle (2pt);
    \draw[densely dashed] (-1.0,1.45)--(1.0,1.45);
    \node at (0,-.46) {even integer: downward cusp};
  \end{scope}
  \begin{scope}[xshift=6cm]
    \draw[->,gray!75] (-1.2,0)--(1.25,0) node[right] {$t$};
    \draw[->,gray!75] (0,-.04)--(0,1.85);
    \draw[lower,very thick,domain=-1:-.002,samples=70]
      plot (\x,{.43+.78*sqrt(abs(\x))});
    \draw[lower,very thick,domain=.002:1,samples=70]
      plot (\x,{.43+.78*sqrt(abs(\x))});
    \fill[lower] (0,.43) circle (2pt);
    \draw[densely dashed] (-1.0,.43)--(1.0,.43);
    \node at (0,-.46) {odd integer: upward cusp};
  \end{scope}
\end{tikzpicture}
\caption{The cusps.}
\label{fig:cusps}
\end{figure}
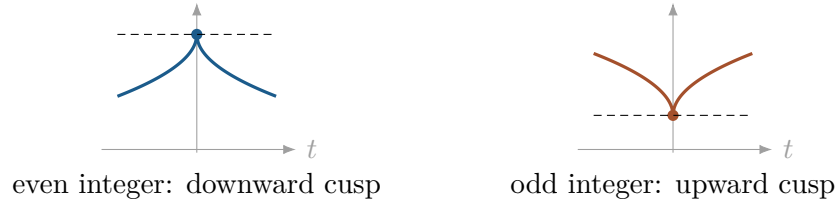

Let $P\subset\R^2$ be an open convex polygon.
An open convex polygon is always understood to be bounded.
Write $P=\bigcap_{i=1}^q\{\ell_i>0\}$, where the affine $\ell_i$
give inward distances from the supporting lines of the sides.
Thus, $|D\ell_i|=1$ and
\[
d_P(z):=\dist(z,\partial P)=\min_i\ell_i(z)\qquad \text{for }z\in P.
\]
For $z\in P$, set
\[
b_P(z)=\prod_{i=1}^q\ell_i(z)^2.
\]
It is clear that $b_P$, $Db_P$, and $b_P/d_P^2$ are bounded in $P$. 
Choose
\[
M_P=1+\sup_P b_P+\sup_P|Db_P|+\sup_P(b_P/d_P^2).
\]
For $\ep>0$, set
\[
a^\ep_P=\frac{\varepsilon b_P}{M_P}.
\]
When there is no confusion, we write $a^\ep_P$ as $a_P$.
We have, for $z\in P$,
\begin{equation}\label{eq:amplitude}
0<a_P(z)\le\varepsilon\min\{1,\dist(z,\partial P)^2\},
\qquad |Da_P(z)|\le\varepsilon.
\end{equation}

The following lemma provides the key iterative mechanism of the construction.
Starting from a pair $(g, P)$ consisting of a function and a polygon, it shows
how to subdivide $P$ by parallel cuts into smaller open convex polygons and
perturb $g$ to a function $\widetilde g$, with $\widetilde g=g$ on $\partial P$,
so that on each resulting subpolygon the pair $(\widetilde g,P_j)$, with a
suitable new orthonormal pair of directions, satisfies properties analogous
to those of $(g,P)$.
\begin{lem}\label{lem:polygon}
Let $P\subset\R^2$ be an open convex polygon, and $g\in C(\overline P)\cap C^\infty(P)$. 
Suppose that $(e,v)$ is a given orthonormal pair of vectors.
Write $A(z)=D_e g(z)$ and $B(z)=D_vg(z)$ for $z\in P$.
Assume, for some $L\ge2$ and $C\ge0$,
\begin{equation}\label{eq:parent}
\begin{cases}
    A(z) \ge L \qquad &\text{ for } z\in P,\\
    A(z)\to\infty\qquad &\text{as }\dist(z,\partial P)\to0,\\
    |B(z)| \leq C A(z)\qquad &\text{ for } z\in P.
\end{cases}
\end{equation}
Given $T>L$, $\ep \in (0,1/16)$, and $N_0\in \N$, one can find an integer $N\ge N_0$ such that: for
\[
\widetilde g(z)=g(z)+a_P(z)w(Nv\cdot z),
\]
each nonempty strip polygon
\[
P_j=P\cap\{j<Nv\cdot z<j+1\}
\]
has a constant orthonormal pair $(e_j,v_j)$ for which
\begin{equation}\label{eq:child}
\begin{cases}
D_{e_j}\widetilde g\ge T \qquad &\text{in }P_j,\\
D_{e_j}\widetilde g(z)\to\infty \qquad &\text{as 
$\dist(z,\partial P_j)\to0$ from inside $P_j$},\\
|D_{v_j}\widetilde g|
 \le8(C+1)D_{e_j}\widetilde g\qquad &\text{in }P_j.
\end{cases}
\end{equation}

At $z\in P$ with $Nv\cdot z=j\in\mathbb Z$, let $\alpha(z)=D_eg(z)+w(j)D_ea_P(z)$.
If $j$ is even, then
\begin{equation}\label{eq:even-jets}
D^+\widetilde g(z)=\{p:p\cdot e=\alpha(z)\},
\qquad D^-\widetilde g(z)=\emptyset.
\end{equation}
If $j$ is odd, then
\begin{equation}\label{eq:odd-jets}
D^-\widetilde g(z)=\{p:p\cdot e=\alpha(z)\},
\qquad D^+\widetilde g(z)=\emptyset.
\end{equation}
Every $p$ in the nonempty semidifferential satisfies
\begin{equation}\label{eq:cut-bound}
|p|\ge A(z)-\varepsilon.
\end{equation}
\end{lem}

\begin{proof}
Fix $j$, and we work within $P_j$ from now on (see Figure \ref{fig:polygons}).
We have $\sigma(z)=\operatorname{sgn}w'(Nv\cdot z)$ is constant in $P_j$.
Set $Q(z)=a_P(z)N|w'(Nv\cdot z)|$.
Then, 
\[
D\widetilde g(z)=A(z)e+B(z)v+w(Nv\cdot z)Da_P(z)+\sigma(z) Q(z)v.
\]

Pick $\delta=\frac1{4(C+1)}$ and set
\[
e_j=\frac{e+\sigma\delta v}{\sqrt{1+\delta^2}},\qquad
v_j=\frac{-\sigma\delta e+v}{\sqrt{1+\delta^2}}.
\]
These two vectors are orthonormal and are suitable perturbations of $e$ and $v$.
Since $|B|\le CA$, $|w|<1$, and $|Da_P|\le\varepsilon$,
\begin{align}
D_{e_j}\widetilde g
&\ge\frac{(1-\delta C)A+\delta Q}{\sqrt{1+\delta^2}}-\varepsilon
 \ge\frac12(A+\delta Q),\label{eq:tilted-lower}\\
|D_{v_j}\widetilde g|
&\le\frac{(C+\delta)A+Q}{\sqrt{1+\delta^2}}+\varepsilon
 \le(C+1)A+Q
 \le8(C+1)D_{e_j}\widetilde g.\label{eq:tilted-other}
\end{align}

\begin{figure}[H]
\centering
\begin{tikzpicture}[>=Latex,scale=.94,font=\small]
  \coordinate (A) at (.05,.5);
  \coordinate (B) at (4.1,.5);
  \coordinate (C) at (5.0,1.2);
  \coordinate (D) at (4.7,3.35);
  \coordinate (E) at (.85,3.2);
  \coordinate (F) at (-.2,2.25);
  \begin{scope}
    \clip (A)--(B)--(C)--(D)--(E)--(F)--cycle;
    \fill[gray!6] (-.5,0) rectangle (5.5,3.7);
    \draw[upper,very thick] (-.5,1.12)--(5.5,1.12);
    \draw[lower,very thick,dashed] (-.5,1.78)--(5.5,1.78);
    \draw[upper,very thick] (-.5,2.44)--(5.5,2.44);
    \draw[lower,very thick,dashed] (-.5,3.10)--(5.5,3.10);
  \end{scope}
  \draw[thick] (A)--(B)--(C)--(D)--(E)--(F)--cycle;
  \draw[->] (1.0,.79)--(1.85,.79) node[right] {$e$};
  \draw[->] (1.0,.79)--(1.0,1.43) node[above] {$v$};
  \draw[->,upper] (2.4,1.58)--(3.15,1.38) node[right] {$e_j$};
  \draw[->,lower] (2.4,1.91)--(3.15,2.08) node[right] {$e_{j+1}$};
  \node at (5.32,3.32) {$P$};
  \node[upper,right] at (5.13,2.44) {even};
  \node[lower,right] at (5.13,1.78) {odd};
\end{tikzpicture}
\caption{Parallel cuts divide a parent polygon into smaller convex
polygons.}
\label{fig:polygons}
\end{figure}
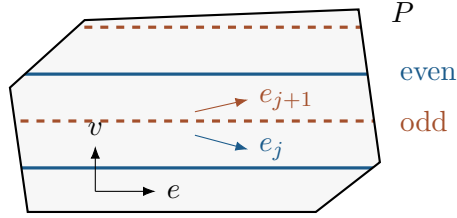
The boundary hypothesis on $A$ makes $K_T:=\{z\in P:A(z)\le2T\}$ a compact subset of $P$.
If it is nonempty, choose an integer
\[
N\ge \max\left\{N_0,\frac{2T}{\delta\min_{K_T}a_P}\right\}.
\]
If $K_T$ is empty, simply choose $N= N_0$.
On $K_T$, \eqref{eq:wave-slope} gives $Q\ge a_P N$, hence
\[
D_{e_j}\widetilde g\ge\delta Q/2\ge T.
\]
Outside $K_T$, $D_{e_j}\widetilde g\geq A/2>T$. 
This proves \eqref{eq:child}. 

If a point in $P_j$ approaches the old boundary $\partial P$, \eqref{eq:tilted-lower} and $A\to\infty$ apply. 
If it approaches a new cut inside $P$, $a_P$ tends to a positive number and $|w'|\to\infty$.
Together with compactness of $\overline{P_j}$, these two observations prove the asserted blowup as $\dist(z,\partial P_j)\to0$.

Fix $z\in P$ with $Nv\cdot z=j$ and set $g_j(z)=g(z)+w(j)a_P(z)$. 
For $h=se+tv$,
\begin{equation}\label{eq:cut-expansion}
\widetilde g(z+h)-\widetilde g(z)
=g_j(z+h)-g_j(z)
 +a_P(z+h)\bigl(w(j+Nt)-w(j)\bigr).
\end{equation}
When $t=0$, the last term vanishes. 
Testing $s>0$ and $s<0$ shows that any $p$ in either semidifferential must satisfy $p\cdot e=D_eg_j(z)=\alpha(z)$.

If $j$ is even, \eqref{eq:wave-cusp} and $a_P(z)>0$ give, for
all sufficiently small $h$, some $\beta>0$ such that the
last term in \eqref{eq:cut-expansion} is at most $-\beta\sqrt{|t|}$. 
For every $p$ with $p\cdot e=\alpha(z)$, Taylor's expansion gives
\[
\widetilde g(z+h)-\widetilde g(z)-p\cdot h
\le O(|h|^2)+|(Dg_j(z)-p)\cdot v|\,|t|
             -\beta\sqrt{|t|}
\le O(|h|^2).
\]
Thus, $p\in D^+\widetilde g(z)$. 
Along $h=tv$, the corresponding first-order quotient tends to $-\infty$, which implies that $D^-\widetilde g(z)=\emptyset$. 
For odd $j$, the conclusion follows in the same way.

Finally, $|w(j)D_ea_P|\le\varepsilon$ and $A\ge2$ imply
\[
|p|\ge |p\cdot e|=\alpha(z)\ge A(z)-\varepsilon.
\]

\end{proof}

\subsection{The planar functions}

\begin{prop}\label{prop:planar}
For $\Om=(0,1)^2$ there exists $f\in C(\overline\Om)$ such
that $\cG^+(f;\Om)$ and $\cG^-(f;\Om)$ are nonempty, disjoint,
and closed relative to $\Om\times\R^2$.
\end{prop}

We construct $f$ by iterating Lemma \ref{lem:polygon}. At each stage, every
current polygon is subdivided into finitely many smaller convex polygons, and
the function is perturbed so that the same structural properties hold on each
child polygon. Importantly, the perturbed function agrees with the previous
one on the boundary of each parent polygon; hence all previously created edges
remain unchanged at later stages. Repeating this process produces an increasing polygonal skeleton
formed by all cuts created during the construction, while the corresponding
functions converge to a limit $f$ with the desired properties.
The iterative construction is related to that in \cite{LST26}.

\begin{proof}
We divide the proof into several steps.

\medskip

\noindent {\bf Step 1 (Iterative constructions).}
For $n\ge0$, let $\varepsilon_n=2^{-n-4}$. 
We begin with
\begin{equation}\label{eq:f0}
f_0(z_1,z_2)
=\arcsin(2z_1-1)+\arcsin(2z_2-1),
\qquad S_0=\partial\Om.
\end{equation}
This function is continuous on $\overline\Om$ and smooth in $\Om$.
Its partial derivatives are
\[
\frac{\partial f_0}{\partial z_i}
  =\frac1{\sqrt{z_i(1-z_i)}}\ge2,\qquad i=1,2.
\]
With $e_0=(1,1)/\sqrt2$, $v_0=(-1,1)/\sqrt2$ and $L_0=2$, we have
\[
D_{e_0}f_0\ge2\sqrt2\ge L_0,\qquad
|D_{v_0}f_0|\le D_{e_0}f_0,
\]
and $D_{e_0}f_0(z)\to\infty$ as
$\dist(z,\partial\Om)\to0$. Thus Lemma~\ref{lem:polygon}
applies to the square $\Om$.

The induction is as follows.
For $n\in \N$, let $S_{n-1}$ be the union of all edges of all convex polygons created at step $n-1$.
Every connected component $P$ of $\Om\setminus S_{n-1}$ is a convex polygon on which $f_{n-1}$ is smooth and satisfies \eqref{eq:parent}, with
$L_{n-1}=n+1$ and some pair $(e_P,v_P)$ and constant $C_P$.
Apply Lemma~\ref{lem:polygon} separately on each $P$ with
\[
\varepsilon=\varepsilon_n,\qquad T=n+2,\qquad
N_0=\max\{n,4\}.
\]
Choose \(N_P\) as in Lemma 3.1, taking the least admissible integer.  Set
\[
f_n(z)=
\begin{cases}
f_{n-1}(z)+a_P^{\ep_n}(z)w(N_Pv_P\cdot z),\qquad&z\in P,\\
f_{n-1}(z),\qquad&z\in S_{n-1}.
\end{cases}
\]
Inside each $P$, the distance to $S_{n-1}$ equals the
distance to $\partial P$. Indeed, $\partial P\subset S_{n-1}$,
and a path from a point in $P$ to any point of $S_{n-1}$
must meet $\partial P$. The bound \eqref{eq:amplitude}, together
with $|w|<1$, gives
\begin{equation}\label{eq:increment}
|f_n(z)-f_{n-1}(z)|
\le\varepsilon_n\min\{1,\dist(z,S_{n-1})^2\},
\qquad z\in\overline\Om.
\end{equation}
Note that $a_P^{\ep_n}(z)w(N_Pv_P\cdot z)=0$ for $z\in \partial P$, and thus, $f_n$ is continuous on $\overline\Om$. 
We also note that $S_n$ consist of $S_{n-1}$ and the closures of all nonempty
cuts
\begin{equation}\label{eq:stage-cuts}
P\cap\{z:N_Pv_P\cdot z=j\},\qquad j\in\mathbb Z,
\end{equation}
in all parent polygons. 
See Figure \ref{fig:nested-partitions}.

\begin{figure}[H]
\centering
\begin{tikzpicture}[>=Latex,font=\small]
  \def\side{3.52}
  \def\hOne{0.61}
  \def\hTwo{0.203333}
  \def\turnOne{0.125}
  \def\turnTwo{0.015}
  \foreach \offset/\stage in {0/1,4.45/2,8.90/3}{
    \begin{scope}[shift={(\offset,0)}]
      \fill[gray!3] (0,0) rectangle (\side,\side);
      \begin{scope}
        \clip (0,0) rectangle (\side,\side);
        \begin{scope}[rotate=45]
          \ifnum\stage>1
            \foreach \j in {-5,...,4}{
              \pgfmathsetmacro{\sgn}{ifthenelse(mod(abs(\j),2)==0,-1,1)}
              \pgfmathsetmacro{\lo}{\j*\hOne}
              \pgfmathsetmacro{\hi}{(\j+1)*\hOne}
              \begin{scope}
                \clip (-1,\lo) rectangle (6,\hi);
                \foreach \k in {-4,...,7}{
                  \pgfmathtruncatemacro{\m}{3*\j+\k}
                  \pgfmathsetmacro{\base}{\m*\hTwo}
                  \ifnum\stage>2
                    \pgfmathsetmacro{\nextsgn}{ifthenelse(mod(abs(\m),2)==0,-1,1)}
                    \begin{scope}
                      \clip (-1,{\base-3.5*\sgn*\turnOne})
                         -- (6,{\base+3.5*\sgn*\turnOne})
                         -- (6,{\base+\hTwo+3.5*\sgn*\turnOne})
                         -- (-1,{\base+\hTwo-3.5*\sgn*\turnOne}) -- cycle;
                      \foreach \ell in {1,2}{
                        \pgfmathsetmacro{\level}{\base+\ell*\hTwo/3}
                        \pgfmathsetmacro{\slope}{\sgn*\turnOne+\nextsgn*\turnTwo}
                        \ifodd\ell
                          \draw[lower!62!white,line width=.32pt]
                            (-1,{\level-3.5*\slope})
                            -- (6,{\level+3.5*\slope});
                        \else
                          \draw[upper!62!white,line width=.32pt]
                            (-1,{\level-3.5*\slope})
                            -- (6,{\level+3.5*\slope});
                        \fi
                      }
                    \end{scope}
                  \fi
                  \ifodd\m
                    \draw[lower!78!white,line width=.65pt,densely dashed]
                      (-1,{\base-3.5*\sgn*\turnOne})
                      -- (6,{\base+3.5*\sgn*\turnOne});
                  \else
                    \draw[upper!78!white,line width=.65pt]
                      (-1,{\base-3.5*\sgn*\turnOne})
                      -- (6,{\base+3.5*\sgn*\turnOne});
                  \fi
                }
              \end{scope}
            }
          \fi
          \foreach \j in {-4,...,4}{
            \ifodd\j
              \draw[lower,line width=1.0pt,densely dashed]
                (-1,{\j*\hOne}) -- (6,{\j*\hOne});
            \else
              \draw[upper,line width=1.0pt]
                (-1,{\j*\hOne}) -- (6,{\j*\hOne});
            \fi
          }
        \end{scope}
      \end{scope}
      \draw[black!80,line width=1.25pt] (0,0) rectangle (\side,\side);
      \node[font=\normalsize] at (1.76,3.88) {$S_{\stage}$};
    \end{scope}
  }
  \draw[->,gray!65] (3.73,1.76)--(4.25,1.76);
  \draw[->,gray!65] (8.18,1.76)--(8.70,1.76);
  \draw[upper,line width=1.1pt] (.20,-.54)--(.70,-.54);
  \node[anchor=west] at (.84,-.54) {downward cusps};
  \draw[lower,line width=1.1pt,densely dashed] (4,-.54)--(4.5,-.54);
  \node[anchor=west] at (4.64,-.54) {upward cusps};
  \draw[black!80,line width=1.1pt] (8.2,-.54)--(8.7,-.54);
  \node[anchor=west] at (8.84,-.54) {$S_0=\partial\Omega$};
  \node[font=\normalsize] at (6.21,-1.50)
    {$S_\infty=\displaystyle\bigcup_{n\ge0}S_n,
      \qquad \overline{S_\infty}=\overline\Omega$};
\end{tikzpicture}
\caption{Successive partitions of $\Omega=(0,1)^2$ (schematic).}
\label{fig:nested-partitions}
\end{figure}
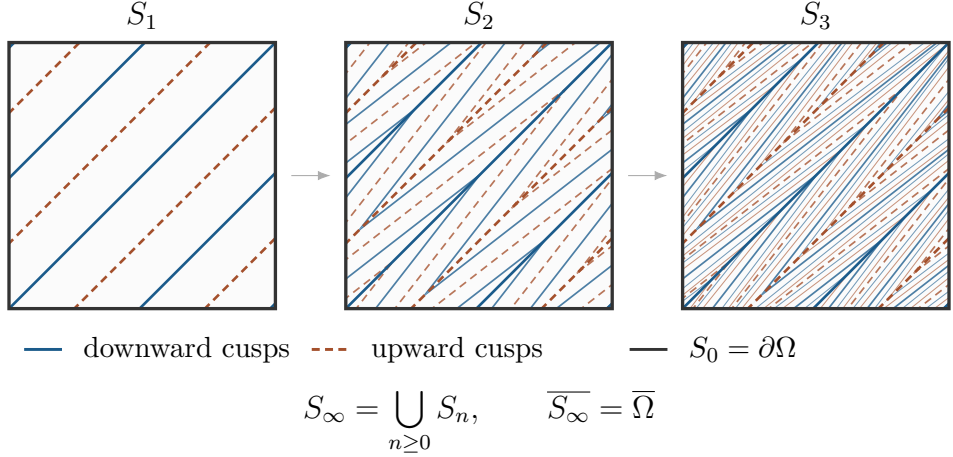

\medskip

\noindent {\bf Step 2 (the limit and its first-order differentials).}
Since $\sum_n\varepsilon_n<\infty$, $\{f_n\}$ converges uniformly to some $f\in C(\overline\Om)$. 
More precisely, as $S_n\subset S_{m-1}$ when $m>n$,
we sum \eqref{eq:increment} and obtain, for $z\in \Omega$,
\begin{equation}\label{eq:tail}
|f(z)-f_n(z)|
\le \left(\,\sum_{m=n+1}^\infty\varepsilon_m\right)\dist(z,S_n)^2=2^{-n-4}\dist(z,S_n)^2.
\end{equation}
If $z\in S_n\cap\Om$, then $f(z)=f_n(z)$ and, for $z+h\in \Omega$,
\[
|f(z+h)-f_n(z+h)|\le2^{-n-4}|h|^2.
\]
Consequently, for $z\in S_n\cap\Om$,
\begin{equation}\label{eq:preservation}
D^\pm f(z)=D^\pm f_n(z).
\end{equation}
Let
\[
S_\infty=\bigcup_{n\ge0}S_n.
\]
We now identify the first-order semidifferentials of the limit function $f$.
We will prove the following dichotomy:

\medskip
\noindent
(1) If $z\in S_\infty\cap\Om$, then exactly one of $D^+f(z)$ and
$D^-f(z)$ is an affine line, while the other is empty;

\medskip
\noindent
(2) If $z\in\Om\setminus S_\infty$, then
\[
D^+f(z)=D^-f(z)=\emptyset.
\]

In fact, if $z\in S_\infty\cap\Om$, let $k\ge1$ be the smallest number such that $z\in S_k$.
Then $z$ is inside a parent polygon $P$ at step $k-1$ and lies on one of its cuts \eqref{eq:stage-cuts}.
Endpoints of new segments belong to the old parent boundary
and hence to $S_{k-1}$, so $z$ is not such an endpoint. 
By \eqref{eq:even-jets}--\eqref{eq:odd-jets}, exactly one of $D^+f(z)$ and $D^-f(z)$ is an affine line, and the other is empty. 
By \eqref{eq:cut-bound}, every $p$ in the nonempty one satisfies
\begin{equation}\label{eq:generation}
|p|\ge L_{k-1}-\varepsilon_k
=k+1-\varepsilon_k,\qquad
|p|\ge D_{e_P}f_{k-1}(z)-\varepsilon_k.
\end{equation}

Now let $z\in\Om\setminus S_\infty$. For $n\geq 0$, let $P_n$ be its step-$n$ polygon such that $z\in P_n$ and $e_n$ the unit direction supplied by the construction. 
Define
\[
d_n=\dist(z,S_n)=\dist(z,\partial P_n).
\]
For $n\geq 1$, the polygon $P_n$ is contained in a strip of width $1/N_P\le1/n$, which implies
\begin{equation}\label{eq:shrinking}
0<d_n\le\frac1{2n}.
\end{equation}
Let $r_n=d_n/2$.
Both segments from $z$ to $z\pm r_ne_n$ lie in $P_n$.
Integrating $D_{e_n}f_n\ge n+2$ on these segments gives
\[
f_n(z+r_ne_n)-f_n(z)\ge (n+2)r_n,\qquad
f_n(z-r_ne_n)-f_n(z)\le-(n+2)r_n.
\]
As $\dist(z\pm r_ne_n,S_n)\le3d_n/2$, the tail bound
\eqref{eq:tail} implies
\begin{align*}
    &f(z+r_ne_n)-f(z)\ge (n+2)r_n - 2^{-n-4}[d_n^2+(3d_n/2)^2] \ge(n+2)r_n - 2^{-n}  r_n^2,\\
    &f(z-r_ne_n)-f(z)\le-(n+2)r_n+2^{-n-4}[d_n^2+(3d_n/2)^2] \le-(n+2)r_n + 2^{-n}  r_n^2.
\end{align*}
For any fixed $p\in\R^2$ it follows that
\begin{align*}
\liminf_{n\to\infty}\frac{f(z+r_ne_n)-f(z)-p\cdot(r_ne_n)}{r_n}
&\ge \liminf_{n\to\infty}((n+2)-|p|-2^{-n}r_n)=+\infty,\\
\limsup_{n\to\infty}\frac{f(z-r_ne_n)-f(z)-p\cdot(-r_ne_n)}{r_n}
&\le\limsup_{n\to\infty}(-(n+2)+|p|+2^{-n}r_n)=-\infty.
\end{align*}
As $p$ is arbitrary, for $z\in\Om\setminus S_\infty$,
\begin{equation}\label{eq:no-jets}
D^+f(z)=D^-f(z)=\emptyset.
\end{equation}

\medskip

\noindent {\bf Step 3 (Closedness of the planar graphs).}
The formula of $S_\infty$ and \eqref{eq:no-jets} show that $\cG^+(f;\Om)$ and $\cG^-(f;\Om)$ are disjoint.
We prove they are relatively closed. 
It is enough to consider $\cG^+(f;\Om)$.
Suppose
\[
z_i\to z\in\Om,\qquad p_i\to p\in\R^2,\qquad
p_i\in D^+f(z_i).
\]
Every $z_i$ lies on $S_\infty$.
Let $k_i\in \N$ be the smallest number such that $z_i \in S_{k_i}$. 
The first bound in \eqref{eq:generation} and the boundedness of $\{p_i\}$ make $\{k_i\}$ bounded. 
Passing to a subsequence, we may fix the step $k$, the parent polygon $P$, and one cut $N_Pv_P\cdot z_i=j$ for all $i$.

In fact, $z\in P$. 
If $z$ lay on $\partial P$, the parent's boundary blow-up and the second bound in \eqref{eq:generation} would give $|p_i|\to\infty$, a contradiction. 
Thus, $z\in P$ and $N_Pv_P\cdot z=j$. 
The fact that $p_i\in D^+f(z_i)$ means $j$ is even, and the exact formula on that cut says
\[
p_i\cdot e_P
=D_{e_P}f_{k-1}(z_i)
  +w(j)D_{e_P}a_P(z_i).
\]
Everything on the right is continuous in $P$.
Passing to the limit gives the same identity for $(z,p)$.
Equations \eqref{eq:even-jets} and \eqref{eq:preservation} now show $p\in D^+f(z)$. 

It is clear from the construction that both graphs are nonempty. 
\end{proof}

\subsection{Proof of the main theorems}
We are ready to prove the main theorems.

\begin{proof}[Proof of Theorem~\ref{thm:main}]
Take $f$ from Proposition~\ref{prop:planar}.
Let
\[
z_0=(\tfrac12,\tfrac12),\qquad
O=(-2,2)^2,\qquad
F(x)=f(z_0+x/4)\qquad \text{for }x\in O.
\]
The affine map $x\mapsto z_0+x/4$ sends $O$ onto
$\Om$, with $\overline{B^2_1}\Subset O$. Directly from
the definitions of the semidifferentials,
\begin{equation}\label{eq:affine-jets}
D^\pm F(x)=\tfrac14D^\pm f(z_0+x/4).
\end{equation}
The map $(x,q)\mapsto(z_0+x/4,4q)$ is a homeomorphism
from $O\times\R^2$ to $\Om\times\R^2$ carrying the two
graphs of $F$ onto those of $f$. Thus they are relatively
closed and disjoint. 
Apply Lemma~\ref{lem:lift} with $m=2$ to conclude.
\end{proof}

\begin{proof}[Proof of Theorem~\ref{thm:main2}]
Let $F$ be as in the proof of Theorem~\ref{thm:main}.  Since $\overline{B^2_1}\Subset O$,  $G^+(F,\overline{B^2_1})$ and $G^-(F,\overline{B^2_1})$ are disjoint closed set on $\R^4$.  Let $K\in C^\infty(\R^4;[1,2])$ be a smooth separation function $G^+(F,\overline{B^2_1})$ and $G^-(F,\overline{B^2_1})$ such that
\[
K=1 \quad \text{on $G^+(F,\overline{B^2_1})$} \quad \mathrm{and} \quad K=2 \quad \text{on $G^-(F,\overline{B^2_1})$}.
\]
It is clear that $F$ is a viscosity solution of
\begin{equation*}
K(x,DF)=c\qquad\text{in }B^2_1
\end{equation*}
for every $c\in[1,2]$. 
\end{proof}

\begin{rem}
By translation and scaling, both theorems hold on any bounded domain of the corresponding dimension. 
The nonautonomous result in Theorem \ref{thm:main2} also extends to the whole space. 
Indeed, let \(\Phi:\mathbb R^2\to B_1^2\) be the smooth diffeomorphism \(\Phi(z)=z/\sqrt{1+|z|^2}\), and set
\[
\widetilde F(z)=F(\Phi(z)),
\qquad
\widetilde K(z,p)=K\bigl(\Phi(z),(D\Phi(z)^{-1})^Tp\bigr).
\]
Then \(\widetilde F\) is a viscosity solution of \(\widetilde K(z,D\widetilde F)=c\) in \(\mathbb R^2\) for every \(c\in[1,2]\). 
\end{rem}
\begin{rem}
The dimension in Theorem \ref{thm:main2} is optimal. Indeed, suppose
that a continuous, nowhere-differentiable function $u$ solves
\[
K(x,u')=c \qquad\text{in }B_1^1
\]
in the viscosity sense, where $K$ is continuous. Since $u$ cannot
be monotone on any open subinterval, its local maximum and local
minimum points are both dense. Constant test functions at these
points give $K(x,0)\le c$ and $K(x,0)\ge c$, respectively.
Continuity of $K$ then yields
\[
K(x,0)=c \qquad\text{for every }x\in B_1^1.
\]

Thus, our lifting construction cannot produce a two-dimensional
autonomous example by starting from a one-dimensional function.
Whether Theorem \ref{thm:main} holds in dimension two or dimension three remains an
interesting question.
\end{rem}

\section*{Use of AI Tools}

The authors used ChatGPT 6 Pro to assist in developing the construction of the pathological solutions through a series of discussions. These discussions were guided by the authors’ ideas and suggestions, including the strategy of constructing disjoint sets through an iterative procedure. The authors verified the mathematical arguments and prepared the manuscript. They take full responsibility for the correctness and content of the paper.

\end{document}